\documentclass[12pt]{amsart}
\usepackage{amsfonts, amssymb, latexsym, graphicx, hyperref, amsmath}
\usepackage{mathtools}
\usepackage[vcentermath]{youngtab}
\usepackage{enumitem}
\usepackage{subcaption}
\usepackage[utf8]{inputenc}
\usepackage[export]{adjustbox}
\usepackage{tikz}
\usetikzlibrary{decorations.markings}
\usepackage{tikz-cd}
\usepackage{dynkin-diagrams}
\usetikzlibrary{calc, shapes, backgrounds,arrows,positioning,decorations.pathmorphing,patterns}

\newtheorem{theorem}{Theorem}[section]
\newtheorem{proposition}[theorem]{Proposition}

\newtheorem{lemma}[theorem]{Lemma}

\newtheorem{claim}[theorem]{Claim}

\newtheorem*{claim*}{Claim}
\newtheorem{corollary}[theorem]{Corollary}
\newtheorem{Main Conjecture}[theorem]{Main Conjecture}
\newtheorem{conjecture}[theorem]{Conjecture}

\newtheorem{problem}[theorem]{Problem}

\theoremstyle{remark}

\newtheorem{remark}[theorem]{Remark}
\theoremstyle{plain}

\newcommand\cD{\operatorname{IP}}
\newcommand\cT{\operatorname{TP}}

\renewcommand{\lg}{\mathfrak{g}}

\newcommand\lgss{{\lg\operatorname{-sat}}}

 \def\cD{\mathcal{D}}

\def\cT{\mathcal{T}}

\newcommand{\cellsize}{11}
\newlength{\cellsz} 
\newsavebox{\cell}
\sbox{\cell}{\begin{picture}(\cellsize,\cellsize)
\put(0,0){\line(1,0){\cellsize}}
\put(0,0){\line(0,1){\cellsize}}
\put(\cellsize,0){\line(0,1){\cellsize}}
\put(0,\cellsize){\line(1,0){\cellsize}}
\end{picture}}
\newcommand\cellify[1]{\def\thearg{#1}\def\nothing{}%
\ifx\thearg\nothing
\vrule width0pt height\cellsz depth0pt\else
\hbox to 0pt{\usebox{\cell} \hss}\fi%
\vbox to \cellsz{
\vss
\hbox to \cellsz{\hss$#1$\hss}
\vss}}
\newcommand\tableau[1]{\vtop{\let\\\cr
\baselineskip -16000pt \lineskiplimit 16000pt \lineskip 0pt
\ialign{&\cellify{##}\cr#1\crcr}}}

\newcommand{\excise}[1]{}%{$\star$\textsc{#1}$\star$}

\renewcommand{\lg}{\mathfrak{g}}
\title{On components of the
tensor square of a Weyl module}
\author{Shiliang Gao}
\address{Dept.~of Mathematics, University of Illinois at Urbana-Champaign, Urbana, IL 61801}
\email{sgao23@illinois.edu}
\author{Dinglong Wang}
\address{Dept.~of Mathematics, University of Illinois at Urbana-Champaign, Urbana, IL 61801}
\email{dw18@illinois.edu}

\date{\today}
\begin{document}
\pagestyle{plain}
\maketitle
\begin{abstract}
    For a simple Lie algebra $\mathfrak{g}$ of type $A_n,B_n,C_n$ or $D_n$, we give a characterization of the set of dominant integral weights $\lambda$ such that for any rational point $\mu$ in the fundamental Weyl chamber, $2\lambda-\mu$ is a non-negative rational combination of the simple roots if and only if $V_{m\mu}\subseteq V_{m\lambda}\otimes V_{m\lambda}$ for some positive integer $m$. 
    %if and only if $2\lambda-\mu$ is a non-negative rational combination of the simple roots. 
\end{abstract}
\section{Introduction}\label{sec:intro}
    
Let $\mathfrak g$ be a simple Lie algebra of rank $n$ over $\mathbb{C}$ with Borel subalgebra $\mathfrak{b}$ and Cartan subalgebra $\mathfrak{h}\subset \mathfrak{b}$. Let $\Phi\subset \mathfrak{h}^*$ be the root system, $W$ be the Weyl group, $\Pi_{\mathfrak{g}} = \{\alpha_1,\ldots,\alpha_n\}$ be its simple roots, $\{\omega_1,\ldots,\omega_n\}$ be the corresponding fundamental weights and $\rho$ be the half sum of positive roots. Let $\Lambda^+$ be the set of dominant integral weights and set $\Lambda_{\mathbb{Q}}^{+} = \Lambda^+\otimes \mathbb{Q}_{\geq 0}$.
%:= \{v\in \mathbb{R}^n: v = \sum_{i = 1}^n a_i \omega_i\ \text{for some }a_i\in \mathbb{R}_{\geq 0}\}$. 

For any $\lambda\in \Lambda^+$, let $V_\lambda$ be the irreducible representation of $\mathfrak{g}$ with highest weight $\lambda$. $V_{\lambda}$ admits a weight space decomposition given by 
$$V_{\lambda} = \bigoplus_{\mu\in \Lambda^+} V_{\lambda}(\mu).$$ 
For $\lambda,\mu \in \Lambda^+$, we write $\lambda\geq \mu$ if the weight space 
$V_{\lambda}(\mu)$ is non-empty.
Equivalently, this is when $\lambda-\mu = \sum_{i = 1}^n c_i \alpha_i$ where $c_i\in \mathbb{Z}_{\geq 0}$ for all $i\in [n]$. We extend this partial order to $\Lambda_{\mathbb{Q}}^+$ where we only require $c_i\in \mathbb{Q}_{\geq 0}$.

%%$\lambda-\mu$ is a $\mathbb{Q}_{\geq 0}$-combination of the simple roots. 
A long-standing conjecture by Kostant states that:
\begin{conjecture}[Kostant]\label{conj:kostant}
    For $\mu\in \Lambda^+, V_{\mu}\subseteq V_{\rho}\otimes V_{\rho}$ if and only if $2\rho \geq \mu$.
\end{conjecture}
In the case of $\mathfrak{g} = \mathfrak{sl}_{n+1}$, Conjecture~\ref{conj:kostant} was proved by Berenstein-Zelevinsky \cite{BZ92}. Chiriv\`i-Kumar-Maffei \cite{CKM17} proved a weakening of Kostant's conjecture. That is, if $\mu\in \Lambda^+$ then $2\rho\geq \mu$ if and only if $V_{m\mu}\subseteq V_{m\rho}\otimes V_{m\rho}$ for some $m\in \mathbb{Z}_{>0}$. In fact, the latter statement is also equivalent to $V_{d\mu}\subseteq V_{d\rho}\otimes V_{d\rho}$ where $d$ is a saturation factor for $\mathfrak{g}$. 
For $\mathfrak{g}$ of type $A_n$, $d = 1$ by the work of Knutson-Tao \cite{KT}. For $\mathfrak{g}$ of type $B_n,C_n$, $d$ can be taken to be $2$ by result of Belkale-Kumar \cite{BK10}, Sam \cite{Sam}, and Hong-Shen \cite{HS15}. For $\mathfrak{g}$ of type $D_n$, $d$ can be taken to be $2$ by work of Sam \cite{Sam}. 
%Here $d$ can be taken to be $1$ if $\mathfrak{g} = \mathfrak{sl}_{n+1}$ by the work of Knutson-Tao \cite{KT}. For $\mathfrak{g}$ of type $B_n,C_n$, $d$ can be taken to be $2$ by result of Belkale-Kumar \cite{BK10} [Sam, HS].
We refer the readers to \cite[Section~10]{K14} for definition and further discussion of the saturation factor. The work of Jeralds-Kumar \cite{JK23} extends the aforementioned work to untwisted affine Kac–Moody Lie algebras and proves Conjecture~\ref{conj:kostant} in affine type A.

In this paper, we study a variation of the result in \cite{CKM17}:
\begin{problem}\label{problem:main}
 Characterize the set of $\lambda\in \Lambda^+$ such that for all $\mu\in \Lambda_{\mathbb{Q}}^+$,
\begin{equation}\label{eqn:mainQ}
    2\lambda \geq \mu \iff V_{m\mu}\subseteq V_{m\lambda}\otimes V_{m\lambda} \text{ for some $m\in \mathbb{Z}_{>0}$}.
\end{equation}
\end{problem}
Our main theorem is a solution to Problem~\ref{problem:main} when $\mathfrak{g}$ is of type $A_n,B_n,C_n$ or $D_n$. 
\begin{theorem}\label{thm:main}
    $\lambda$ satisfy \eqref{eqn:mainQ} if and only if
    \begin{itemize}
        \item $\lambda = N\rho$ for some $N\in \mathbb{Z}_{\geq 0}$ when $\mathfrak{g}$ is of type $A_n$ or $D_n$;
        \item $\lambda = N\rho+k\omega_n$ for $N\in \mathbb{Z}_{\geq 0}$ and $k\in \mathbb{Z}_{\geq -N}$ when $\mathfrak{g}$ is of type $B_n$ or $C_n$.
    \end{itemize}
\end{theorem}
%Following definition in \cite{BJK21}, 

Let %$\mathcal{K}(\mathfrak{g})$ and 
$\mathcal{K}(\lg)_{\mathbb{Q}}$
be the set of rational points in the \emph{Kostka cone} associated to $\mathfrak{g}$:
%that is,
%\[\mathcal{K}(\lg):= \{(\lambda,\nu)\in (\Lambda^+)^2: \lambda\geq \nu\}\]
%and
\[\mathcal{K}(\lg)_{\mathbb{Q}}:= \{(\lambda,\mu)\in (\Lambda_{\mathbb{Q}}^+)^2: \lambda\geq \mu\},\]
%% revision needed, define IP_\lambda as a slice of the Kostka cone seems better.
and define the 
%\emph{tensor semigroup} of $\mathfrak{g}$ to be
%\[\lgs:= \{(\lambda,\mu,\nu)\in (\Lambda^+)^3: V_{\nu}\subseteq V_{\lambda}\otimes V_{\mu}\}\]
%and the 
\emph{saturated tensor cone} to be
\[\lgss := \{(\lambda,\mu,\nu)\in (\Lambda^+_{\mathbb{Q}})^3: mV_{\nu}\subseteq mV_{\lambda}\otimes mV_{\mu} \text{ for some }m\in \mathbb{Z}_{>0}\}.\]
There has been significant interests in understanding the structure of the Kostka cone and the saturated tensor cone (see \emph{e.g.} \cite{BJK21,burrull2023dominant, K14} and the references therein). Theorem~\ref{thm:main} can be understood as studying certain affine slices of the two cones as follows.

%The Kostant's conjecture is then equivalent to the following:
%\begin{conjecture}
%    For any $\nu\in \Lambda^+$, 
%    \[\nu\in \mathcal{K}(\lg)\cap\{(2\rho,\cdot)\}\iff \nu\in \lgs\cap\{(\rho,\rho,\cdot)\}.\]
%    {\color{purple} Shiliang: Bad notation, need to fix this}
%\end{conjecture}
Fix $\lambda\in \Lambda^+$, define the \emph{dominant weight polytope} to be the affine slice of $\mathcal{K}(\mathfrak{g})_{\mathbb{Q}}$ given by
\[\mathcal{D}_{2\lambda} := \{\mu\in \Lambda_{\mathbb{Q}}^+:(2\lambda,\mu)\in \mathcal{K}(\lg)_{\mathbb{Q}}\}.\]
Define similarly the \emph{tensor polytope} to be the affine slice of $\lgss$:
\[\mathcal{T}_{\lambda} = \{\mu\in \Lambda_{\mathbb{Q}}^+: (\lambda,\lambda,\mu)\in \lgss\}.\]
Problem~\ref{problem:main} is then equivalent
 to the following:
\begin{problem}\label{q:main}
    Characterize the set of $\lambda\in \Lambda^+$ such that $\cD_{2\lambda} = \cT_{\lambda}$.
\end{problem}

We note that $\cT_{\lambda}\subseteq \cD_{2\lambda}$ for all $\lambda\in \Lambda^+$. Our strategy is to study the vertices of $\cD_{2\lambda}$ and check if they lie in $\cT_{\lambda}$. The same line of thinking has also been explored in recent work of Boysal \cite{boysal2023kostants}.

%So it is enough to characterize all $\lambda$ such that every vertex of $\cD_{2\lambda}$ is contained in $\cT_{\lambda}$. Our main theorem is an answer to Problem~\ref{q:main} when $\mathfrak{g}$ is of classical type.
%\begin{theorem}
%    Suppose $\mathfrak{g}$ is of classical type, then 
%    \[\cD_{2\lambda} = \cT_{\lambda}\iff \lambda = \begin{cases}
%     N\rho &  \text{if }\mathfrak{g} \text{ is of type $A_n$ or $D_n$}\\
%     N\rho+k\omega_n & \text{if }\mathfrak{g} \text{ is of type $B_n$ or $C_n$}
%     \end{cases},     \]
%     for some $k,N\in \mathbb{Z}$ such that $k\geq -N$.
%\end{theorem}

\section{notation and background}
\subsection{Root system, weights and Weyl groups}\label{sec:root}
Following Section~\ref{sec:intro}, we continue with more background on root systems in classical Lie types.

We adopt the following convention on root systems for classical types:
\begin{itemize}
    \item Type $A_{n-1}$: $\alpha_i = e_i-e_{i+1},\omega_i = \sum_{j = 1}^i e_j$ for $i \in [n-1]$; 
    \item Type $B_n$: $\alpha_i = e_i-e_{i+1},\omega_i = \sum_{j=1}^{i}e_j$ for $i\in [n-1]$, $\alpha_n = e_n$ and $\omega_n = \frac{1}{2}\sum_{j= 1}^n e_j$;
    \item Type $C_n$: $\alpha_i = e_i-e_{i+1}$ for $i\in [n-1]$, $\alpha_n = e_n$ and $\omega_i = \sum_{j=1}^{i}e_j$ for $i\in [n]$;
    \item Type $D_n$: $\alpha_i = e_i-e_{i+1}$ for $i\in [n-1]$, $\alpha_n = e_{n-1}+e_{n}$, $\omega_i = \sum_{j = 1}^i e_j$ for $i\in [n-2]$, $\omega_{n-1} = \sum_{j=1}^{n-1}\frac{1}{2}e_j-\frac{1}{2}e_{n}$ and $\omega_{n} = \frac{1}{2}\sum_{j=1}^n e_j$.
\end{itemize}
Let $\mathfrak{g}$ be of type $A_n$ through $D_n$, using the above convention, we will identify $\lambda = \sum_{i=1}^n\lambda_i e_i\in \Lambda^+$ with the sequence $(\lambda_1,\lambda_2,\ldots,\lambda_n)$. Let $\langle -,-\rangle$ be the inner product on $\mathfrak{h}$ so that the fundamental weight $\omega_i$ is dual to the simple coroots, i.e. $\langle \omega_i,\alpha_j^\vee\rangle = \delta_{i,j}$. 

\subsection{Dominant weight polytopes and their vertices} In \cite{BJK21}, the authors studied the dominant weight polytope $\cD_{\lambda}$ to understand the extremal rays of $\mathcal{K}(\mathfrak{g})$. Here we recall their result on the vertices of $\cD_\lambda$:
\begin{theorem}[\cite{BJK21}, Proposition~3.9]\label{thm:IPvertex}
    For $I\subseteq [n]$, let $W_I$ be the parabolic subgroup of $W$ generated by $\{s_{\alpha_i}:i\in I\}$. Then for any $\lambda\in \Lambda^+$, the vertices of $\cD_\lambda$ is given by
    \[v_I(\lambda) = \frac{1}{|W_{I}|}\sum_{w\in W_I}w\cdot \lambda.\]
\end{theorem}
An immediate corollary is that the set of vertices depends linearly on $\lambda$:
\begin{corollary}[\cite{BJK21}, Corollary~3.11]\label{cor:sum}
    Any vertex of $\cD_{\lambda+\mu}$ is a sum of a vertex of $\cD_{\lambda}$ and a vertex of $\cD_{\mu}$.
\end{corollary}

\subsection{Embedding of Dynkin diagrams and Schubert structure constants}\label{sec:proj}
For any $\Pi'\subset \Pi$, let $\Phi'\subset \Phi_{\mathfrak{g}}$ be the root subsystem generated by $\Pi'$. Let $\mathfrak{g}'\subset \mathfrak{g}$ be the subalgebra defined by
%For a pair $(\lambda,\mu)$ that is not primitive, let $\Pi_0 = \{\alpha_i:c_i\neq 0\}$ be the subset of $\Pi$ consisting of simple roots that appears in \eqref{eqn:lambda-mu} and $\Phi_0$ be the corresponding root subsystem. Let $\mathfrak{g}_0\subset \mathfrak{g}$ be the subalgebra defined by
\begin{equation}\label{eqn:subalg}
    \mathfrak{g}' = \mathfrak{h}'\oplus \bigoplus_{\alpha\in \Phi'} \mathfrak{g}_{\alpha},
\end{equation}
where $\mathfrak{h}'$ is spanned by $\{h_\alpha\in \mathfrak{h}:\alpha \in \Phi'\}$. Let $p:\mathfrak{h}^*\rightarrow (\mathfrak{h}')^*$ be the natural orthogonal projection induced by $\langle-,-\rangle$.

\begin{lemma}[\cite{BZ88}, Proposition~1.3]\label{lemma:multg0}
    Let $\lambda,\nu,\mu$ be dominant integral weights of $\mathfrak{g}$ such that $\lambda+\nu-\mu\in \Phi'$. Denote $c_{p(\lambda),p(\mu)}^{p(\nu)}(\mathfrak{g}')$ the multiplicity of $V_{p(\nu)}\subset V_{p(\lambda)}\otimes V_{p(\mu)}$ as $\mathfrak{g}'$-module. Then
    \[c^{\nu}_{\lambda,\mu}(\mathfrak{g}) = c^{p(\nu)}_{p(\lambda),p(\mu)}(\mathfrak{g}').\]
\end{lemma}
In particular, if $|\lambda|+|\mu| = |\nu|$ and $\lambda,\mu,\nu\in \Lambda^+$ are partitions; that is, $\lambda_i,\mu_i,\nu_i\in \mathbb{Z}_{\geq 0}$ for all $i\in [n]$, then
\begin{equation}\label{eqn:redtosln}
    c_{\lambda,\mu}^\nu(\mathfrak{g}) = c_{\lambda,\mu}^\nu(\mathfrak{sl}_n).
\end{equation}

\subsection{Inclusion of Dynkin diagrams and Schubert structure constants}

Let $G$ be the connected complex semisimple algebraic group with Lie algebra $\mathfrak{g}$. Let $B$ be the Borel subgroup of $G$ with Lie algebra $\mathfrak{b}$. The generalized flag variety $G/B$ has finitely many orbits under the left action of Borel subgroup $B_{-}$. They are indexed by $w\in W$ where $W$ is the Weyl group of $G$. The \emph{opposite Schubert varieties} $Y_w$ are the closure of these orbits. Denote $[Y_w]$ the Poincar\'e dual of the fundamental class of $Y_w$. We have
\[[Y_w]\in H^{\ell(w)}(G/B),\]
where $\ell(w)$ is the Coxeter length of $w$. The set $\{[Y_w]:w\in W\}$ form a $\mathbb{Z}$-basis of the cohomology ring $H^{*}(G/B)$.

Define the \emph{Schubert structure constant} $c^{w}_{u,v}(G/B)$ to be the structure constant of $H^{*}(G/B)$ with respect to the opposite Schubert basis $\{[Y_w]:w\in W\}$: 
\[[Y_u] \cdot [Y_v] = \sum_{w\in W}c^{w}_{u,v}(G/B)[Y_w].\]
%We note that $c_{u,v}^w = 0$ unless $\ell(w) = \ell(u)+\ell(v)$.

Let $H_1$ and $H_2$ be (connected) finite Dynkin diagrams and $\iota:H_1\hookrightarrow H_2$ be an inclusion of diagrams; that is, a graph theoretic injection that respects arrows. Label the nodes of $H_1$ and $H_2$ by their corresponding simple roots 
$$\Pi_{H_1} = \{\alpha_1,\ldots,\alpha_{r(H_1)}\}\text{ and }\Pi_{H_2} = \{\beta_1,\ldots,\beta_{r(H_2)}\}$$
such that $\iota(\alpha_i) = \beta_i$ for all $i\in [r(H_1)]$.

Denote $W_1$ and $W_2$ the Weyl groups corresponding to $H_1$ and $H_2$ respectively. The inclusion $\iota$ then induces an injection 
$\iota: W_1 \hookrightarrow W_2$
by sending $s_{\alpha_i}$ to $s_{\beta_i}$. Let $X_1,X_2$ be the generalized flag varieties corresponding to $H_1,H_2$ respectively.

\begin{proposition}[\cite{ABCD}, Theorem~2.1]\label{prop:inclusion}
For any $u,v,w\in W_1$,
    \[c_{u,v}^w(X_1) = c_{\iota(u),\iota(v)}^{\iota(w)}(X_2).\]
\end{proposition}

\begin{remark}\label{rmk:convention}
    We note that the convention in \cite{ABCD} is different from the one used in \cite{BS00,CKM17}. The former studies the structure constant $c_{u,v}^w$ in the opposite Schubert basis (one where $\ell(w)$ is the codimension) whereas the later two (see also equation~\eqref{eqn:Nov27aaa}) work in the Schubert basis (where $\ell(w)$ is the dimension). The two basis are related by $[X_w] = [Y_{w_0w}]$ where $w_0$ is the longest element in the Weyl group $W$. As a result, Proposition~\ref{prop:inclusion} does not hold if we simply replace $c_{u,v}^w$ in the statement with the structure constant under Schubert basis. However, as we will see later, the inclusion of Dynkin diagram we consider here is an isomorphism and the choice of $u,v$ are self dual. In this case we can apply Proposition~\ref{prop:inclusion} in the setting of \eqref{eqn:Nov27aaa}.
\end{remark}

For the purpose of this paper, we focus on the map of the Dynkin diagrams $\iota: D_3\leftrightarrow A_3$ where we identify $\alpha_1 = e_1{-}e_2,\alpha_2 = e_2{-}e_3,\alpha_3 = e_2{+}e_3$ with $\beta_1 = e_2{-}e_3, \beta_2 = e_1{-}e_2, \beta_3 = e_3{-}e_4$ respectively. 

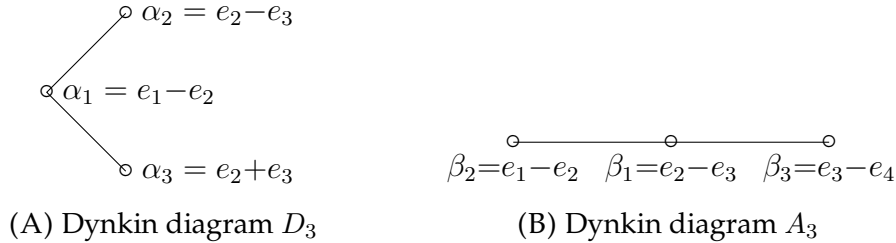
\begin{figure}[h!]
    \centering
    \subcaptionbox{Dynkin diagram $D_3$ \label{fig:D3}}[.4\textwidth]{
            \begin{tikzpicture}[scale = 0.7]
    %\fill [green, opacity  = 0.25] (0,1) rectangle (2,4);

   \draw (1.5,1.5)--(0,0)--(1.5,-1.5);
    \node at (0,0) {$\circ$};
    \node at (1.5,1.5) {$\circ$};
    \node at (1.5,-1.5) {$\circ$};
    \node at (1.75,0) {$\alpha_1 = e_1{-}e_2$};
    \node at (3.25,1.5) {$\alpha_2 = e_2{-}e_3$};
    \node at (3.25,-1.5) {$\alpha_3 = e_2{+}e_3$};
    \end{tikzpicture}}
    \subcaptionbox{Dynkin diagram $A_3$ \label{fig:A3}}[.4\textwidth]{
        \begin{tikzpicture}[scale = 0.7]
    %\fill [green, opacity  = 0.25] (0,1) rectangle (2,4);
    %\fill [green, opacity  = 0.25] (1,4) rectangle (2,2);
    \draw (0,0)--(3,0)--(6,0);
    \node at (0,0) {$\circ$};
    \node at (3,0) {$\circ$};
    \node at (6,0) {$\circ$};
    \node at (0,-0.5) {$\beta_2{=}e_1{-}e_2$};
    \node at (3,-0.5) {$\beta_1{=}e_2{-}e_3$};
    \node at (6,-0.5) {$\beta_3{=}e_3{-}e_4$};
    \end{tikzpicture}}
        %\caption{Caption}
        %\label{fig:u}
        
    \caption{Identification of Dynkin diagrams of type $A_3$ and $D_3$}
    \label{fig:diagram}
        
    \end{figure}
Let $X_1 = SL(4)/B$ and $X_2 = SO(6)/B$ be the generalized flag varieties of type $A_3$ and~$D_3$.

%Consider now the natural inclusion of Dynkin diagram $\iota_{A,B},\iota_{A,C},\iota_{A,D}$ from $H_1$ into $H_2$ where we identify $\alpha_i = e_{i}-e_{i+1}\in \Pi_{A_{n-1}}$ with  $\alpha_i = e_{i}-e_{i+1}\in \Pi_{B_{n}},\Pi_{C_n},\Pi_{D_n}$ for all $i\in [n-1]$. The inclusion induces a natural inclusion of the symmetric group into Weyl groups of type $B_n,C_n$ and $D_n$. 

\subsection{Inequalities defining the saturated tensor cones}

Let $\{x_i:i\in [n]\}\subset \mathfrak{h}$ be the dual to the simple roots $\{\alpha_i:i\in [n]\}$, namely $\alpha_i(x_j) = \delta_{i,j}$. For a maximal parabolic subgroup $P = P_\alpha$, let $x_P:= x_{i_P}$ where $\alpha = \alpha_{i_P}$. For $\mu\in \Lambda^+$, the dual of $\mu$ is $\mu^* = -w_0\mu$. The following characterization of $\lgss$ is due to Berenstein-Sjamaar. 
%on the inequalities defining saturated tensor cones (See also \cite{CKM17} and Theorem 10 of \cite{K14}):
\begin{proposition}[\cite{BS00}]\label{prop: BKineq}
$(\lambda,\mu,\nu)\in \lgss$
%For $\mu\in \Lambda_{Q}^+$, $\mu\in \cT_{\lambda}$
if and only if the following inequality
\begin{equation}\label{eqn:TPineq}
    \lambda\left(u x_P\right)+\mu\left(v x_P\right)+\nu^*\left(w x_P\right) \leq 0.
\end{equation}
holds for all maximal parabolic subgroup $P \subset G$ and all triple $(u,v,w) \in (W^P)^3$ such that product of the corresponding Schubert classes in $G / P$:
\begin{equation}\label{eqn:Nov27aaa}
    \left[X_u^P\right] \cdot\left[X_v^P\right] \cdot\left[X_w^P\right]=k\left[X_e^P\right] \in H^*(G / P, \mathbb{Z}) \text{ for some }k>0.
\end{equation}
%Here we set $x_P:=x_{i_P}$ where $\alpha_{i_P} = \alpha$.
\end{proposition}

In the case where $\mathfrak{g} = \mathfrak{sl}_{n}$, a description of $\lgss$ was conjectured by Horn and proved by combining the result of Klyachko \cite{Klyachko} and of Knutson-Tao \cite{KT}. 
We note that the set of \emph{Horn inequalities} and the set of inequalities \eqref{eqn:TPineq} are the same and we encourage the readers to look at Section~4 of \cite{K14} for a detailed exposition.
We include a description of the Horn inequalities below for both completeness and convenience to use in the proof of Theorem~\ref{thm:main}. 

Let $[n]:=\{1,2,\ldots n\}$.  For any
$I=\{i_1<i_2< \cdots<i_d\}\subseteq [n]$
define the partition
\[
 \tau(I):=(i_d-d\geq \cdots\geq i_2-2 \geq i_1-1).
\]
%This bijects subsets of $[n]$ of cardinality $d$ with partitions whose Young diagrams are contained in a $d\times(n-d)$ rectangle. 
%The following is a combination of results of Klyachko \cite{Klyachko} and of Knutson-Tao \cite{KT}. 

\begin{theorem}[Horn inequalities]\label{thm:classicalHorn}
Let $\mathfrak{g} = \mathfrak{sl}_{n}$ and let $\lambda,\mu,\nu\in \Lambda_{\mathbb{Q}}^+$ be such that $\lambda+\mu-\nu$ lies in the root lattice. 
Then
$(\lambda,\mu,\nu)\in \lgss$ if and only if
for every $d<n$, and every triple of subsets $I,J,K\subseteq [n]$ of cardinality $d$ such that $c_{\tau(I),
\tau(J)}^{\tau(K)}> 0$,
\begin{equation}\label{eq:ineq}
\sum_{k\in K}\nu_k\leq
\sum_{i\in
I}\lambda_i+\sum_{j\in J}\mu_j.\end{equation}
\end{theorem}

Belkale-Kumar \cite{BK06} described a subset of inequalities \eqref{eqn:TPineq} that characterize $\lgss$ using a deformed cohomology product $\odot$ on $H^*(G/P)$. It is proved by Ressayre \cite{R10} proved that the subset is irredundant, namely any proper subset defines a different cone. We refer the readers to Section~6 of \cite{BK06} for the precise definition of $\odot$. 

\begin{theorem}\label{thm:minineq}\cite{BK06,R10}
    Proposition~\ref{prop: BKineq} still holds if we replace \eqref{eqn:Nov27aaa} with
    \begin{equation}\label{eqn:Dec3aaa}
        \left[X_u^P\right] \odot\left[X_v^P\right] \odot\left[X_w^P\right]=\left[X_e^P\right] \in (H^*(G / P, \mathbb{Z}),\odot).
    \end{equation}
    Moreover, replacing \eqref{eqn:Nov27aaa} with \eqref{eqn:Dec3aaa} yields a minimal set of inequalities describing $\lgss$. 
\end{theorem}

In \cite{NL2}, Orelowitz, Yong, and the first author introduced \emph{extended Horn inequalities} in their study of tensor product multiplicities in the stable range for $\mathfrak{g} = \mathfrak{so}_{2n+1},\mathfrak{sp}_{2n}$ and $\mathfrak{so}_{2n}$. They contain the Horn inequalities \eqref{eq:ineq} and the Belkale-Kumar inequalities when $\mathfrak{g} = \mathfrak{sp}_{2n}$ (as in Theorem~\ref{thm:minineq}) as special cases (see Section~3 of \cite{NL3}). Together with Ressayre in \cite{NL3}, they proved the following:

\begin{theorem}[Extended Horn inequalities]
  \label{thm:GOYconj}
Let $\mathfrak{g} = \mathfrak{sp}_{2n}$ and  $\lambda,\mu,\nu\in\Lambda_{\mathbb Q}^+$. Then $(\lambda,\mu,\nu)\in
\lgss$ if and only if
\begin{equation}
\label{Ineq:EH}
0\leq \sum_{a\in A}\lambda_a-\sum_{a'\in A'}\lambda_{a'}+\sum_{b\in B}\mu_b-\sum_{b'\in B'}\mu_{b'}+\sum_{c\in C}\nu_c-\sum_{c'\in C'}\nu_{c'}
\end{equation}
for any subsets $A,A',B,B',C,C'\subset [n]$ such that
\begin{enumerate}
\item $A \cap A'= B \cap B' = C \cap C' = \emptyset$;
\item \label{cond:size}$|A|+|A'|=|B|+|B'|=|C|+|C'|=|A'|+|B'|+|C'|=:r$;
\item \label{cond:ind6} the Littlewood-Richardson coefficients
$c_{\alpha_1,\alpha_2}^{\tau(A)},c_{\alpha_2,\alpha_3}^{\tau(C')},c_{\alpha_3,\alpha_4}^{\tau(B)},c_{\alpha_4,\alpha_5}^{\tau(A')}, c_{\alpha_5,\alpha_6}^{\tau(C)},c_{\alpha_6,\alpha_1}^{\tau(B')}>0$ for some partitions $\alpha_1,\dots,\alpha_6$.
%$(\tau(A), \tau(C'), \tau(B), \tau(A'), \tau(C), \tau(B'))\in \operatorname{NL}^6\!\operatorname{-sat}(r)$.
\end{enumerate}
\end{theorem}

%\begin{proposition}\label{prop:Horn}
    %%(la,mu,nu)\in \lgss imples they satisfy Horn inequalities. Need extra assumption that $|la|+|mu| = |nu|$ in type D
%\end{proposition}

\section{Proof of Theorem~\ref{thm:main}}
Fix $\lambda = (\lambda_1,\ldots,\lambda_n)\in \Lambda^+$. We begin with a useful lemma that holds in type $A$ through $D$.

%\noindent $(\implies):$ We will prove the contrapositive statement; that is, if $\lambda$ is not of the form in Theorem~\ref{thm:main}, then $TP_{\lambda}\subsetneq IP_{2\lambda}$. Our strategy is to construct a vertex $\mu\in IP_{2\lambda}$ and show that $(\lambda,\lambda,\mu)$ does not satisfy a defining inequality of $\lgss$.

\begin{lemma}\label{lemma:ap}
    If $\lambda_m+\lambda_{m+2} \neq 2\lambda_{m+1}$ for some $m\in [n-2]$ then $ \cT_{\lambda}\neq \cD_{2\lambda}$.
\end{lemma}
\begin{proof}
Let $\mu = v_{\{m,m+1\}}(\lambda)$ as in Theorem~\ref{thm:IPvertex}. More explicitly,
\begin{equation}\label{eqn:defmu}
    \mu_j = \begin{cases}
\frac{2(\lambda_m+\lambda_{m+1}+\lambda_{m+2})}{3} & \text{ if }m\leq j\leq m+2\\
2\lambda_j & \text{ otherwise }
\end{cases}.
\end{equation}
Notice that for any $\mathfrak{g}$ of classical Lie type, $\mu$ is a vertex of $\cD_{2\lambda}$ and $|\mu| = 2|\lambda|$. 

We first study the case where $\mathfrak{g} = \mathfrak{sl}_n$.
Consider the inequalities 
\begin{equation}\label{eqn:Horn1}
    \sum_{k\in [m+2]\setminus \{m\}} \mu_k \leq \sum_{i\in [m+2]\setminus \{m+1\}} \lambda_i +\sum_{j\in [m+2]\setminus \{m+1\}} \lambda_j,
\end{equation}
and 
\begin{equation}\label{eqn:Horn2}
    \sum_{k\in [m+2]\setminus \{m,m+1\}} \mu_k \leq \sum_{i\in [m+2]\setminus \{m,m+2\}} \lambda_i +\sum_{j\in [m+2]\setminus \{m,m+2\}} \lambda_j.
\end{equation}
Observe that both \eqref{eqn:Horn1} and \eqref{eqn:Horn2} are cases of Horn inequalities \eqref{eq:ineq} for the triple $(\lambda,\lambda,\mu)$. Indeed, in the case of \eqref{eqn:Horn1}, we have $\tau(K) = (1,1), \tau(I) = \tau(J) = (1)$, and in \eqref{eqn:Horn2}, we have $\tau(K) = (2), \tau(I) = \tau(J) = (1)$. In both cases, $c_{\tau(I),\tau(J)}^{\tau(K)} = 1>0$. 
Since $\mu_j = 2\lambda_j$ for all $j\notin [m,m+2]$, we can rewrite \eqref{eqn:Horn1} and \eqref{eqn:Horn2} as 
\begin{equation}\label{eqn:Horn3}
    \mu_{m+1}+\mu_{m+2}\leq 2(\lambda_{m}+\lambda_{m+2})\tag{7'}
\end{equation}
and 
\begin{equation}\label{eqn:Horn4}
    \mu_{m+2}\leq 2\lambda_{m+1}.\tag{8'}
\end{equation}
By \eqref{eqn:defmu}, Equation~\eqref{eqn:Horn3} holds if and only if $\lambda_m+\lambda_{m+2}\geq 2\lambda_{m+1}$, and Equation~\eqref{eqn:Horn4} holds if and only if $\lambda_m+\lambda_{m+2}\leq 2\lambda_{m+1}$. Since $\lambda_m+\lambda_{m+2}\neq 2\lambda_{m+1}$,
\eqref{eqn:Horn3} and \eqref{eqn:Horn4} cannot hold simultaneously, and thus $\mu\notin \cT_{\lambda}$. Since $\mu$ is a vertex of $\cD_{2\lambda}$, we have $\cT_{\lambda} \neq \cD_{2\lambda}$ when $\mathfrak{g} = \mathfrak{sl}_n$.

For $\mathfrak{g} = \mathfrak{so}_{2n+1},\mathfrak{sp}_{2n}$ or $\mathfrak{so}_{2n}$, we use Lemma~\ref{lemma:multg0}. Let $\lambda\in \Lambda^+$ and let $\mu$ be constructed as in \eqref{eqn:defmu}. Notice that
$\mu$ is always a vertex of the dominant weight polytope $\cD_{2\lambda}$, and 
\[\sum_{i\in [n]}2\lambda_i = \sum_{i\in [n]}\mu_i.\]
Since there is some $d\in \mathbb{Z}_{> 0}$ such that $d\lambda_i$ and $d\mu_i$ are integers for all $i\in [n]$, $d\lambda+d\lambda-d\mu$ lies in the root lattice of $\mathfrak{sl}_n$. By \eqref{eqn:redtosln},
$$c_{dk\lambda,dk\lambda}^{dk\mu}(\mathfrak{g}) =c_{dk\lambda,dk\lambda}^{dk\mu}(\mathfrak{sl}_n) = 0,$$
for all $k\in \mathbb{Z}_{> 0}$. As a result, $dk\mu\notin \cT_{dk\lambda}$ for any $\mathfrak{g}$ of classical lie type. Since the inequalities defining $\cT_{\lambda}$ (See equation~\eqref{eqn:TPineq}) are linear, we can conclude that $\mu\notin \cT_{\lambda}$.
\end{proof}

\noindent \emph{Proof of Theorem~\ref{thm:main}:} We first prove the $(\implies)$ direction. By Lemma~\ref{lemma:ap}, we are left to show that for $\mathfrak{g} = \mathfrak{so}_{2n}$, $\cD_{2\lambda} \neq \cT_{\lambda}$ for all $\lambda = N\rho+k\omega_n$ with $k\neq 0$. 
\begin{claim}\label{claim:0vertex}
    $0$ is a vertex of $\cD_{2\lambda}$.
\end{claim}

\noindent \emph{Proof of Claim~\ref{claim:0vertex}: }This follows directly from Theorem~\ref{thm:IPvertex} by setting $I = [n]$. \qed

We wish to show that $0\notin \cT_{\lambda}$. By Proposition~\ref{prop: BKineq}, it is enough to show the existence of $u,w\in W^P$ such that 
\[\lambda(ux_P)>0 \text{ and }[X_u^P]\cdot [X_u^P] \cdot [X_w^P]  = k[X_e^P]\in H^*(SO(2n)/B) \text{ for some }k>0.\]
Consider first the case $n =3$. Let $P$ be the maximal parabolic subgroup associated to the simple root $\alpha_1 = e_1-e_2$ as in Figure~\ref{fig:D3}. Define $u_1 = s_3s_1, u_2 = s_2s_1\in W^P$ and set $w_0^P$ be the minimal $W_P$-coset representative of the longest permutation $w_0$. In one line notation, $u_1 = -3\;1\,{-}2$ and $u_2 = 312$. Let $X_1 = SO(6)/B$ and $X_2 = SL(4)/B$.
\begin{claim}\label{claim:Nov26aaa}
    $c_{u_i,u_i}^{w_0^P}(X_1) = 1$ and thus $[X_{u_i}^P]\cdot [X_{u_i}^P]\cdot [X_{w_0^P}^P] = [X_e^P]\in H^{*}(SO(6)/P)$ for $i\in \{1,2\}$.
\end{claim}

\noindent \emph{Proof of Claim~\ref{claim:Nov26aaa}:} Consider the identification of Dynkin diagrams of type $A_3$ and $D_3$ as in Figure~\ref{fig:diagram}. Then $\iota(u_1) = s_3s_2 = 1423\in S_{4}, \iota(u_2) = s_1s_2 = 2314\in S_{4}$ and $\iota(w_0^P) = 3412\in S_4$. Since the Weyl group of type $A_3$ and $D_3$ are isomorphic under the map $\iota$, we have $\iota(w_0) = w_0$ and thus the difference in convention as mentioned in Remark~\ref{rmk:convention} is not an issue. We can then apply Proposition~\ref{prop:inclusion} to get 
\[c_{u_1,u_1}^{w_0^P}(X_1) = c_{1423,1423}^{3412}(X_2) = 1 \]
and
\[c_{u_2,u_2}^{w_0^P}(X_1) = c_{2314,2314}^{3412}(X_2) = 1\]
Since $u_1,u_2,w_0^P\in W^P$, we have
\[[Y_{u_i}^P]\cdot [Y_{u_i}^P]\cdot [Y_e^P] = [Y_{w_0^P}^P]\in H^*(SO(6)/P)\text{ for }i\in \{1,2\}.\]
where $[Y_u^P]$ is the opposite Schubert class dual to $[X_u^P]$ for $u\in W^P$. Notice that both $[Y_{u_i}^P]$ are both self dual. We can therefore conclude that 
\[\pushQED{\qed} [X_{u_i}^P]\cdot [X_{u_i}^P]\cdot [X_{w_0^P}^P] = [X_e^P]\text{ for }i\in \{1,2\}. \qedhere\popQED\]

Suppose $\lambda = N\rho+k\omega_n$ for some nonzero $k\in\mathbb{Z}_{\geq -N}$. If $k>0$, then
\[\lambda(u_2x_P) = \langle (2N+\frac{k}{2}, N+\frac{k}{2}, \frac{k}{2}),(0,0,1)\rangle = \frac{k}{2}>0.\]
If $k<0$, then
\[\lambda(u_1x_P) = \langle (2N+\frac{k}{2}, N+\frac{k}{2}, \frac{k}{2}),(0,0,-1)\rangle = -\frac{k}{2}>0.\]
Therefore $\cD_{2\lambda}\neq \cT_\lambda$ if $k\neq 0$ and we conclude the ($\implies$) direction in type $D_3$. For general $n$, consider the projection as in Section~\ref{sec:proj} where $\mathfrak{g}'$ is of type $D_3$. For $\lambda = N\rho+k\omega_n$ with $k\in \mathbb{Z}_{\geq-N}$, if $c_{m\lambda,m\lambda}^{0}>0$, then by Lemma~\ref{lemma:multg0}, 
\[c_{mp(\lambda),mp(\lambda)}^{0}(\mathfrak{g}')>0.\]
By Lemma~3.1 of \cite{gao2022degrees}, as a dominant integral weight of $\mathfrak{g}'$, $p(\lambda) = N\omega_1+N\omega_2+(N+k)\omega_3$ is of the form $(2N+\frac{k}{2}, N+\frac{k}{2}, \frac{k}{2})$ as seen above. Therefore $$c_{m\lambda,m\lambda}^0(\mathfrak{g}) = c_{mp(\lambda),mp(\lambda)}^0(\mathfrak{g'})= 0 \text{ unless }k = 0.$$
By Claim~\ref{claim:0vertex}, we conclude that  $\cD_{2\lambda}\neq \cT_{\lambda}$ for all $\lambda \neq N\rho$ when $\mathfrak{g} = \mathfrak{so}_{2n}$.

$(\impliedby):$ We divide into two cases:

\noindent \textbf{Case I} ($\mathfrak{g} = \mathfrak{sl}_{n}\text{ or }\mathfrak{so}_{2n}$):
By Corollary~\ref{cor:sum}, if $\mu$ is a vertex of $\cD_{2N\rho}$, then 
\[\mu = \sum_{i = 1}^{N}\mu^{(i)},\]
where each $\mu^{(i)}$ is a vertex of $\cD_{2\rho}$. In particular, $2\rho\geq \mu^{(i)}$ for all $i\in [N]$. By Theorem~3 of \cite{CKM17}, $\mu^{(i)}\in \cT_{\rho}$ and thus inequalities in \eqref{eqn:TPineq} are satisfied. Since these inequalities are linear, $(N\rho,N\rho,\mu)$ also satisfy all inequalities in \eqref{eqn:TPineq}. Therefore $\mu\in \cT_{N\rho}$ and we conclude that $\cD_{2N\rho} = \cT_{N\rho}$ for all $N\in \mathbb{Z}_{>0}$. 

\medskip

\noindent \textbf{Case II} ($\mathfrak{g} = \mathfrak{so}_{2n+1}\text{ or }\mathfrak{sp}_{2n}$):
By Theorem~4.1 and 4.2 of \cite{KS07} (See also Corollary~7.5 of \cite{K10}), the saturated tensor cone and thus the tensor polytope are the same for $\mathfrak{g} = \mathfrak{so}_{2n+1}$ and $\mathfrak{sp}_{2n}$ where we embed the root lattice of $\mathfrak{sp}_{2n}$ into the root lattice of $\mathfrak{so}_{2n+1}$ by sending $e_i$ to $e_i$ using the coordinate as defined in Section~\ref{sec:root}.

We first prove that if $\lambda = \rho-\omega_n = (n-1,n-2,\dots,1,0)$, then $\cD_{2\lambda} = \cT_{\lambda}$. The proof here mainly follows the proof strategy of Theorem~3 in \cite{CKM17}. Let $W$ be the signed permutation group on $[n]$. This is the Weyl group for $\mathfrak{g} = \mathfrak{so}_{2n+1}$ or $\mathfrak{sp}_{2n}$. For $v\in W$, let $v_1\dots v_n$ where $v_i = v(i)$ be the one-line notation of $u$. Let $P = P_r, v\in W_P$ and $\mu\in \Lambda_{\mathbb{Q}}^+$, then 
\begin{equation}\label{eqn:Dec4aaa}
    \mu(vx_P) = \sum_{i = 1}^r \mu_{v_i},
\end{equation}
where we set $\mu_i := -\mu_{-i}$ if $i<0$. Let $(u,v,w)$ be any triple satisfying \eqref{eqn:Dec3aaa}.

\begin{claim}\label{claim:omegan}
    $(\omega_n+u^{-1}\omega_n+v^{-1}\omega_n+ w^{-1}\omega_n)(x_P)= 0$
\end{claim}
\noindent \emph{Proof of Claim~\ref{claim:omegan}:}
We prove this claim in the case $\mathfrak{g} = \mathfrak{sp}_{2n}$. The case $\mathfrak{so}_{2n+1}$ follows immediately as the $n$-th fundamental weight is exactly half the $n$-th fundamental weight of $\mathfrak{sp}_{2n}$ under the natural embedding of the root lattice.
Since $u,v,w$ satisfy \eqref{eqn:Dec3aaa}, by Theorem~\ref{thm:GOYconj}, the inequality \eqref{eqn:TPineq}
is an extended Horn inequality and can be written in the form of \eqref{Ineq:EH}. Since the longest element $w_0 \in W$ is $-1$, $\nu^* = -w_0\nu = \nu$ for all $\nu\in \Lambda^+$. Using \eqref{eqn:Dec4aaa}, we can write 
\[A = \{u_i: i\in [r], u_i<0\}, A' = \{u_i: i\in [r],u_i>0\},\]
and similarly for the sets $B,B',C,C'$. Therefore %if $\mathfrak{g} = \mathfrak{sp}_{2n}$, then
\[(u^{-1}\omega_n)(x_P) = |A'|-|A|,(v^{-1}\omega_n)(x_P) = |B'|-|B|\text{ and } (w^{-1}\omega_n)(x_P) = |C'|-|C|.\]
%and if $\mathfrak{g} = \mathfrak{so}_{2n+1}$, then 
%\[(u^{-1}\omega_n)(x_P) = \frac{1}{2}(|A'|-|A|),(v^{-1}\omega_n)(x_P) = \frac{1}{2}(|B'|-|B|)\text{ and } (w^{-1}\omega_n)(x_P) = \frac{1}{2}(|C'|-|C|).\]
By condition~(\ref{cond:size}) in Theorem~\ref{thm:GOYconj}, 
\[\pushQED{\qed}(\omega_n+u^{-1}\omega_n+v^{-1}\omega_n+ w^{-1}\omega_n)(x_P) = r+|A'|+|B'|+|C'|-|A|-|B|-|C| = 0.\qedhere \popQED\]

By Equation~(5) of \cite{CKM17} 
\[(\rho+u^{-1}\rho+v^{-1}\rho+ w^{-1}\rho)(x_P)\leq 0\]
for all triples $(u,v,w)$ satisfying  \eqref{eqn:Dec3aaa}.
Combining with Claim~\ref{claim:omegan}, we have
\begin{equation}\label{eqn:lambda}
    (\lambda+u^{-1}\lambda+v^{-1}\lambda+ w^{-1}\lambda)(x_P)\leq 0,
\end{equation}
with $\lambda = \rho-\omega_n$ as defined above. Let $\mu = v_{I}(2\lambda)$ be any vertex of $\cD_{2\lambda}$. Here we will not distinguish $\mathfrak{so}_{2n+1}$ and $\mathfrak{sp}_{2n}$ since their corresponding Weyl groups are the same and thus the vertices of the dominant weight polytope $\cD_{2\lambda}$ are the same by Theorem~\ref{thm:IPvertex}.
\begin{claim}\label{claim:mu}
    $\lambda(ux_P)+\lambda(vx_P)+\mu^*(wx_P)\leq (\lambda+u^{-1}\lambda+v^{-1}\lambda+w^{-1}\lambda)(x_P)$. 
\end{claim}
\noindent \emph{Proof of Claim~\ref{claim:mu}:} Since $\mu^* = \mu$, it is enough to show that $w^{-1}(\mu-\lambda) \leq \lambda$. Let $W_I\subseteq W$ be the subgroup generated by $\{s_{\alpha_i}:i\in I\}$ and let $w_0^I$ be the longest element in $W_I$. Since $\lambda = (n-1,n-2,\dots,1,0)$, by Theorem~\ref{thm:IPvertex}, 
\[\mu-\lambda = w_0^I\lambda.\]
Since $\lambda\in \Lambda^+$, 
\[\pushQED{\qed}w^{-1}(\mu-\lambda) = w^{-1}w_0^I(\lambda)\leq \lambda.\qedhere \popQED\]
Combining Claim~\ref{claim:mu} and \eqref{eqn:lambda}, $(\lambda,\lambda,\mu)$ satisfy inequalities in \eqref{eqn:TPineq}. Therefore $(\lambda,\lambda,\mu)\in \lgss$ and $\cD_{2\lambda} = \cT_{\lambda}$ when $\lambda = \rho-\omega_n$. 

We now prove $\cD_{2\omega_n} = \cT_{\omega_n}$ in the case $\mathfrak{g} = \mathfrak{sp}_{2n}$. The $\mathfrak{so}_{2n+1}$ case follows from the same reasoning. Let $\mu$ be any vertex of $\cD_{2\omega_n}$, by Theorem~\ref{thm:IPvertex}, $\mu = (2,\dots,2,0,\dots,0) = 2\omega_k$ for some $k\in [n]$. To prove $\mu\in \cT_{\omega_n}$, it is enough to show that $(\omega_n,\omega_n,2\omega_k)$ satisfy the Extended Horn inequalities as in \eqref{Ineq:EH}. Indeed,  
\begin{align*}
    0&\leq  |C'|+|C'|+2|C\cap[k]|-2|C'\cap [k]|\\
    &= |A|-|B'| + |B| - |A'| + 2|C\cap[k]|-2|C'\cap [k]|\\
    &= |A|-|A'| + |B| - |B'| + 2|C\cap[k]|-2|C'\cap [k]|\\
    &= \sum_{a\in A}(\omega_n)_a - \sum_{a'\in A'}(\omega_n)_{a'} + \sum_{b\in B}(\omega_n)_b - \sum_{b'\in B'}(\omega_n)_{b'} + \sum_{c\in C}(2\omega_k)_c - \sum_{c'\in C'}(2\omega_k)_{c'}
\end{align*}
as we desired. 

We now finish the ($\impliedby$) direction when $\mathfrak{g} = \mathfrak{so}_{2n+1}$ or $\mathfrak{sp}_{2n}$. For any $\lambda$ of the form $N\rho+k\omega_n$ where $k\in \mathbb{Z}_{\geq -N}$, we can write $\lambda = N(\rho-\omega_n)+ (k+N)\omega_n$. Since $\cD_{2(\rho-\omega_n)} = \cT_{\rho-\omega_n}, \cD_{2\omega_n} = \cT_{\omega_n}$, by Corollary~\ref{cor:sum} and linearity of the defining inequalities of $\lgss$, we conclude that $\cD_{2\lambda} = \cT_{\lambda}$. \qed

\section*{acknowledgements}
This project was initiated as part of the Illinois Combinatorics Lab for Undergraduate Experiences (ICLUE) Representation theory and Combinatorics program in summer 2022. 
We would like to thank Alexander Yong for stimulating conversations and for organizing the summer program. We also thank Sam Jeralds for helpful comments and references. SG was partially supported
by NSF RTG grant DMS 1937241 and NSF Graduate Research Fellowship under
grant No. DGE-1746047.
\bibliographystyle{abbrv}
\bibliography{ref.bib}

\end{document}